\documentclass[12pt,a4paper,reqno]{amsart}
\allowdisplaybreaks
\usepackage{amsmath}
\usepackage{amsfonts}
\usepackage{amssymb,amsthm,amsfonts,amsthm,latexsym,enumerate,url,cases}
\numberwithin{equation}{section}
\usepackage{mathrsfs}
\usepackage{hyperref}
\hypersetup{colorlinks=true,citecolor=blue,linkcolor=blue,urlcolor=blue}
\theoremstyle{plain}
\newtheorem{theorem}{Theorem}

\newtheorem{lemma}{Lemma}
\newtheorem{problem}{Problem}

\theoremstyle{definition}

\usepackage{etoolbox}
\makeatletter
\patchcmd{\@settitle}{\uppercasenonmath\@title}{}{}{}
\patchcmd{\@setauthors}{\MakeUppercase}{}{}{}
\patchcmd{\section}{\scshape}{}{}{}
\makeatother

\begin{document}

\title
[{Solutions to some problems on unique representation bases}]
{Solutions to some problems on unique representation bases}

\author
[Y. Ding] 
{Yuchen Ding$^*$}

\address{(Yuchen Ding) School of Mathematical Sciences,  Yangzhou University, Yangzhou 225002, People's Republic of China}
\email{ycding@yzu.edu.cn}
\thanks{$^*$ ORCID:  0000-0001-7016-309X}

\keywords{unique representation bases, Sidon sets, growth of sets, inductive process}
\subjclass[2010]{11B13, 11B34, 11B05}

\begin{abstract}
In this note, three 2003 problems of Nathanson and two 2007 problems of Chen on unique representation bases for the integers are resolved.
\end{abstract}
\maketitle

\section{Introduction}
Let $\mathbb{N}$ be the set of natural numbers and $A$ a subset of $\mathbb{N}$. For any $n\in \mathbb{N}$, let
$$
r_A(n)=\#\big\{(a,a'):n=a+a',a\le a',a,a'\in A\big\}.
$$
One of the most attractive unsloved problems in additive combinatorics is the Erd\H os-Tur\'an conjecture on additive bases \cite{erdos-turan} posed in 1941, which states that if $r_A(n)\ge 1$ for all sufficiently large $n$, then $r_A(n)$ cannot be unbounded. 

However, things get changed if one replaces the natural set $\mathbb{N}$ by the integer set $\mathbb{Z}$. A set $A$ is called a {\it unique representation basis} of $\mathbb{Z}$ if $r_A(n)=1$ for all $n\in \mathbb{Z}$. In a fairly original article, Nathanson \cite{Nathanson} showed that there is a unique representation basis of $\mathbb{Z}$. Nathanson further proved that a unique representation basis can be arbitrarily sparse.  In addition, Nathanson proved that every unique representation basis $A$ for the integers satisfies $A(-x,x)\le \sqrt{8x}$, where $A(y,x)=\big|A\cap [y,x]\big|$. Nathanson himself constructed a unique representation basis $A$ such that 
$$
A(-x,x)\ge (2/\log 5)\log x+0.63.
$$
Based on the results above, Nathanson asked whether there exists a number $\theta<1/2$ so that $A(-x,x)\le x^\theta$ for every unique representation basis $A$ and for all sufficiently large $x$? 
Nathanson \cite{Nathanson} also posed the following three open problems:

\begin{problem}\label{problem1}
For each number $c>2/\log 5$, does there exist a unique representation basis $A$ such that $A(-x,x)>c\log x$ for all sufficiently large $x$?
\end{problem}

\begin{problem}\label{problem2}
Does there exist a unique representation basis $A$ such that
$$
\lim_{x\rightarrow\infty}\frac{A(-x,x)}{\log x}=\infty?
$$
\end{problem}

\begin{problem}\label{problem3}
Does there exist a number $\theta>0$ and a unique representation $A$ such that $A(-x,x)\ge x^{\theta}$ for all sufficiently large $x$?
\end{problem}

As was already shown by Chen \cite{chen}, the classical results on finite Sidon sets (see e.g. \cite{Sidon}) can be used to construct unique representation basis $A$ with large growth of the number of its elements infinitely often. Precisely, Chen proved that for any $\varepsilon>0$ there is a unique representation basis $A$ for the integers such that $A(-x,x)\ge x^{1/2-\varepsilon}$ for infinitely many positive integers $x$, which gave a negative answer to the problem of Nathanson mentioned above.
Chen \cite{chen} then posed some open problems. Two of them are as follows: 

\begin{problem}\label{problem4}
Does there exist a real number $c>0$ and a unique representation basis $A$ such that
$A(-x,x)\ge c\sqrt{x}$ for infinitely many positive integers $x$?
\end{problem}

\begin{problem}\label{problem5}
Does there exist a real number $c>0$ and a unique representation basis $A$ such that
$A(-x,x)\ge c\sqrt{x}$ for all real numbers $x\ge 1$?
\end{problem}

In this article, we shall solve the Problems \ref{problem1} to \ref{problem5} above.

\begin{theorem}\label{thm1}
There exists a unique representation $A$ such that $A(-x,x)\ge \frac{1}{8}x^{1/3}$ for all sufficiently large $x$.
\end{theorem}
Theorem \ref{thm1} answers affirmatively Problems \ref{problem1}, \ref{problem2} and \ref{problem3} from Nathanson.

\begin{theorem}\label{thm2}
Let $\varepsilon>0$ be an arbitraily small number. Then there is a unique representation basis $A$ such that
$A(-x,x)\ge \left(\frac{\sqrt{2}}{2}-\varepsilon\right)\sqrt{x}$ for infinitely many $x$.
\end{theorem}
Theorem \ref{thm2} answers affirmatively Problem \ref{problem4} of Chen.

\begin{theorem}\label{thm3}
Let $A$ be a unique representation basis. Then
$$
\liminf_{x\rightarrow\infty}\frac{A(-x,x)}{\sqrt{x/\log x}}\le 4\sqrt{7}.
$$
\end{theorem}
Theorem \ref{thm3} answers negatively Problem \ref{problem5} of Chen. 

At present, I cannot solve the following problem of Chen \cite{chen}.

{\it Does there exist a real number $\theta<1/2$ such that for any unique representation basis $A$ there are infinitely many positive integers $x$ with $A(-x,x)\le x^{\theta}$?}

Now, let $\mathscr{A}$ be set of all unique representation bases $A$ such that there exists some real number $c>0$ so that $A(-x,x)>c x^{1/2}$ for infinitely many $x$. For any $A\in \mathscr{A}$, let 
$$
c_A:=\limsup_{x\rightarrow\infty}\frac{A(-x,x)}{\sqrt{x}} \quad \text{and} \quad c_{\mathscr{A}}:=\sup_{A\in \mathscr{A}}c_A.
$$
It is clear from Theorem \ref{thm2} that $\mathscr{A}$ is nonempty
and 
$$
c_{\mathscr{A}}\ge \frac{\sqrt{2}}{2}.
$$
As we mentioned above, Nathanson's result implies that
$$
c_{\mathscr{A}}\le 2\sqrt{2}.
$$
Thus, it would be of some interests to ask `{\it what is the exact value of $c_{\mathscr{A}}$}'? It seems safe to guess $c_{\mathscr{A}}\ge 1$.

\section{Proofs}

The idea leading to the solutions of Nathanson's problems benefits from an application of the greedy algorithm together with the inductive process of Chen \cite{chen}.

\begin{proof}[Proof of Theorem \ref{thm1}]
Firstly, we will construct by inductive process a series of sets 
$$
A_1\subset A_2\subset\cdots \subset A_i\subset\cdots
$$
satisfying the following five requirements

I. $r_{A_h}(n)\le 1$ for any $n\in \mathbb{Z}$,

II. $r_{A_h}(n)=1$ for any $n\in \mathbb{Z}$ with $|n|\le h-1$,

III. $\big|a^*_{h}\big|<\big|a^*_{h+1}\big|\le 64 \big|a^*_{h}\big|$,

IV. $A_{h}\big(-\big|a^*_h\big|,\big|a^*_h\big|\big)\ge \frac{1}{2}\big(\big|a^*_h\big|\big)^{1/3},$

V. $0\not\in A_h$\\
for any positive integer $h$, where $a^*_{h}$ is the element of $A_{h}$ with maximum absolute value.

Let $A_1=\{-1,1\}$ and $x_1=1$. Clearly, $A_1$ satisfies I, II, IV, V. Assume that we have already constructed the sets $A_1\subset A_2\subset\cdots \subset A_{h}$ satisfying conditions I to V. Let $m$ be the integer with minimum absolute value so that $r_{A_{h}}(m)=0$. Then by inductive hypothesis II and the definition of $m$ we have
\begin{align}\label{eq-1-1}
h\le |m|\le 2\big|a^*_{h}\big|+1.
\end{align}
Take 
\begin{align}\label{eq-1-2}
b=4\big|a^*_{h}\big|+|m|
\end{align}
and 
$$
B=A_{h}\cup\big\{-b,b+m\big\}.
$$
Note that $m=(b+m)+(-b)\in B+B$ and the following four sets
$$
2A_h, \quad A_{h}-b, \quad A_{h}+b+m, \quad \big\{m, -2b,2b+2m\big\}
$$
are disjoint, where 
$$
A_{h}-b:=\big\{a-b:a\in A_{h}\big\}
$$
and
$$
A_{h}+b+m:=\big\{a+b+m:a\in A_{h}\big\}.
$$
So, we have $r_B(n)\le 1$ for all $n\in \mathbb{Z}$, $r_B(m)=1$ and $0\not\in B$. If $r_B(-m)=0$, then take
\begin{align}\label{eq-1-3}
\widetilde{b}=4b+5|m|
\end{align}
and 
$$
\widetilde{B}=B\cup\big\{-\widetilde{b},\widetilde{b}-m\big\}.
$$
It can be seen that $r_{\widetilde{B}}(n)\le 1$ for all $n\in \mathbb{Z}$, $r_{\widetilde{B}}(-m)=1$ and $0\not\in \widetilde{B}$ via similar discussions above.
Let
\begin{align*}
D_{h+1}=
\begin{cases}
B, & \text{if~} r_B(-m)=1;\\
\widetilde{B}, & \text{otherwise.}
\end{cases}
\end{align*}
Then we know that $\big|a^*_{h}\big|<\big|d^*_{h+1}\big|\le 64 \big|a^*_{h}\big|$ from (\ref{eq-1-1}), (\ref{eq-1-2}) and (\ref{eq-1-3}), where $d^*_{h}$ is the element of $D_{h+1}$ with maximum absolute value.
Now, we will add a few new elements into $D_{h+1}$ to form $A_{h+1}$, making it satisfied with the dense condition, i.e., inductive hypothesis IV. This can be done via a greedy algorithm which will be put in the next paragraph. 

If 
\begin{align}\label{eq-1-4}
\big|D_{h+1}\big|\ge \frac{1}{2}\big(\big|d^*_{h+1}\big|\big)^{1/3},
\end{align}
then we stop here and let $A_{h+1}=D_{h+1}$. Clearly, we have 
$$
\big|a^*_{h}\big|<\big|a^*_{h+1}\big|=\big|d^*_{h+1}\big|\le 64 \big|a^*_{h}\big|
$$
and hence 
$$
A_{h+1}\big(-\big|a^*_h\big|,\big|a^*_h\big|\big)=\big|D_{h+1}\big|\ge \frac{1}{2}\big(\big|a^*_h\big|\big)^{1/3}.
$$
Therefore, the inductive hypotheses I to V are satisfied with this $A_{h+1}$. 
Now, suppose the contrary of (\ref{eq-1-4}), we consider the following set 
\begin{align*}
W_{h+1}:=\left\{d_1+d_2-d_3,~\frac{d_4+d_5}{2}: d_j\in D_{h+1}, ~1\le j\le 5\right\}.
\end{align*}
We add integers $|n|\le \big|d^*_{h+1}\big|/2$ step by step. First, we choose a nonzero integer $n_1$ with minimum absolute value which does not belong to $W_{h+1}$ and add it into $D_{h+1}$ to form a new set $D^1_{h+1}$. We claim that $D^1_{h+1}$ satisfies  $r_{D^1_{h+1}}(n)\le 1$ for any $n\in \mathbb{Z}$. Actually, if $r_{D^1_{h+1}}(n)\ge 2$ for some $n$, then we must have
$$
n_1+d_3=d_1+d_2 \quad \text{or} \quad 2n_1=d_4+d_5 
$$
for some $d_j\in D_{h+1} ~(1\le j\le 5)$ which contradicts the selection of $n_1$.
Next, we 
consider the following set 
\begin{align*}
W^1_{h+1}:=\left\{d_1+d_2-d_3, ~\frac{d_4+d_5}{2}: d_j\in D^1_{h+1},~1\le j\le 5\right\}.
\end{align*}
Again, we choose a nonzero integer $n_2$ with minimum absolute value which does not belong to $W^1_{h+1}$ and add it into $D^1_{h+1}$ to form a new set $D^2_{h+1}$. We continue this process and stop it until the $g$-th step when we have
$$
\big|D^g_{h+1}\big|\ge \frac{1}{2}\big(\big|d^*_{h+1}\big|\big)^{1/3}.
$$
This can be done because otherwise
$$
\big|W^g_{h+1}\big|\le \big|D^g_{h+1}\big|^3+\big|D^g_{h+1}\big|^2\le  \frac{1}{8}\big|d^*_{h+1}\big|+\frac{1}{4}\big|d^*_{h+1}\big|^{2/3}<\frac{1}{2}\big|d^*_{h+1}\big|,
$$
which can be continued again by the process above. Suppose now that the process stops at the $g$-th step, then we take $A_{h+1}=D^g_{h+1}$. We still have 
$$
\big|a^*_{h}\big|<\big|a^*_{h+1}\big|=\big|d^*_{h+1}\big|\le 64 \big|a^*_{h}\big|.
$$
Moreover, we have
$$
\big|A_{h+1}\big|=\big|D^g_{h+1}\big|\ge \frac{1}{2}\big(\big|d^*_{h+1}\big|\big)^{1/3}.
$$
It is clear from the constructions above that the new set $A_{h+1}$ satisfies the inductive hypotheses I to V.

Let
\begin{align*}
A=\bigcup_{h=1}^{\infty}A_h
\end{align*}
By the inductive hypothesis I and  II, we have $r_{A_{h}}(n)=1$ for any $n\in \mathbb{Z}$, which means that $A$ is a unique representation basis for the integers. Let $x$ be a sufficiently large number. Then there exists some $h$ so that 
\begin{align*}
\big|a^*_{h}\big|\le x<\big|a^*_{h+1}\big|<64 \big|a^*_{h}\big|,
\end{align*}
from which it follows that
\begin{align*}
A(-x,x)\ge A_h(-x,x)\ge A_h\big(-\big|a^*_{h}\big|,\big|a^*_{h}\big|\big)\ge\frac{1}{2}\big(\big|a^*_h\big|\big)^{1/3}\ge \frac{1}{8}x^{1/3},
\end{align*}
where the last but one inequality comes from inductive hypothesis IV.
\end{proof}

Before presenting the proof of Theorem \ref{thm2}, we introduce an important object from additive combinatorics. 

A set 
$$
S=\{s_1<s_2<\cdot\cdot\cdot<s_t\}\subset \{1,2,\cdots,n\}
$$ 
is called a {\it Sidon set} if all the sums 
$$
s_i+s_j~(i\le j)
$$ 
are different. Let $F_2(n)$ be the largest cardinality of a Sidon set in $\{1,2,\cdots,n\}$.
The second lemma is a standard result involving Sidon sets \cite[Theorem 7, page 88]{Sidon}.

\begin{lemma}\label{lem2}
We have
$$
F_2(n)=n^{1/2}+O\big(n^{1/3}\big).
$$
\end{lemma}
The error term in Lemma \ref{lem2} is not the best known one at present (see e.g. \cite{O'}), but the relation $F_2(n)\sim n^{1/2}$ is actually applicable enough for our purpose.

\begin{proof}[Proof of Theorem \ref{thm2}]
We will firstly construct by inductive process a series of sets 
$$
A_1\subset A_2\subset\cdots \subset A_i\subset\cdots
$$
and a series of positive integers
$$
x_1<x_2<\cdots<x_i<\cdots
$$
satisfying the following four constraints

I. $r_{A_h}(n)\le 1$ for any $n\in \mathbb{Z}$,

II. $r_{A_{2h}}(n)=1$ for any $n\in \mathbb{Z}$ with $|n|\le h$,

III. $A_{2h-1}(-x_h,x_h)\ge \left(\frac{\sqrt{2}}{2}-\varepsilon\right)\sqrt{x_h},$

IV. $0\not\in A_h$\\
for any positive integer $h$.

Let $A_1=\{-1,1\}$ and $x_1=1$. Assume that we have already constructed the sets $A_1\subset A_2\subset\cdots \subset A_{2h-1}$ and $x_1<x_2<\cdots<x_h$ satisfying conditions I to IV. Let $m$ be the integer with minimum absolute value so that $r_{A_{2h-1}}(m)=0$. We claim that $|m|\ge h$. In fact, if $h=1$, then $|m|=1=h$. For $h>1$, we have $r_{A_{2h-2}}(m)=0$ since $A_{2h-2}\subset A_{2h-1}$. Then by inductive hypothesis II we have $|m|\ge h$. Similarly to the constructions of 
$B$ and $\widetilde{B}$ in the proof of Theorem \ref{thm1}, 
we let
$$
B=A_{2h-1}\cup\big\{-b,b+m\big\} \quad 
\text{and} \quad
\widetilde{B}=B\cup\big\{-\widetilde{b},\widetilde{b}-m\big\},
$$
where $b=4\big|a^*_{2h-1}\big|+|m|$ and $\widetilde{b}=4b+5|m|$. Next, we let
\begin{align*}
A_{2h}=
\begin{cases}
B, & \text{if~} r_B(-m)=1;\\
\widetilde{B}, & \text{otherwise.}
\end{cases}
\end{align*}
Then, the set $A_{2h}$ satisfies inductive hypotheses I, II and IV. It remains to construct a set $A_{2h+1}$ and an integer $x_{h+1}>x_h$ so that $A_{2h}\subset A_{2h+1}$ and the hypotheses I, III, IV hold. Improving the former one given by Chen \cite[Lemma 2]{chen}, we shall make use of the full strength of Lemma \ref{lem2} and we will put it in the next paragraph. 

Denote by $a^*_{2h}$ the element of $A_{2h}$ with maximum absolute value. Now, let $y$ be a (positive) parameter to be decided later. By Lemma \ref{lem2} there is a set $S\subset \big[0,\big(1/2-\varepsilon/2\big)y\big)$ with
\begin{align}\label{eq-2-1}
\big|S\big|\ge \sqrt{\left(1/2-\varepsilon/2\right)y}+O\big(y^{1/3}\big)>\left(\frac{\sqrt{2}}{2}-\frac{\varepsilon}{2}\right)\sqrt{y}+O\big(y^{1/3}\big)
\end{align}
such that $r_S(n)\le 1$ for all $n\in \mathbb{Z}$, where $\varepsilon$ is a sufficiently small number. Let 
$$
\widetilde{S}=S+y/2:=\big\{s+y/2:s\in S\big\}.
$$
Then $\widetilde{S}\subset \left[y/2,\big(1-\varepsilon/2\big)y\right)$ satisfies $r_{\widetilde{S}}(n)\le 1$ for all $n\in \mathbb{Z}$ and
\begin{align}\label{eq-2-2}
\big|\widetilde{S}\big|=\big|S\big|\ge \left(\frac{\sqrt{2}}{2}-\frac{\varepsilon}{2}\right)\sqrt{y}+O\big(y^{1/3}\big)
\end{align}
from (\ref{eq-2-1}). Noting that
$$
A_{2h}-A_{2h}\subset\big[-2\big|a^*_{2h}\big|,2\big|a^*_{2h}\big|\big],
$$
we know that the number of pairs $s_1\neq s_2\in \widetilde{S}$ so that
$$
s_1-s_2\in A_{2h}-A_{2h}
$$
is not exceeding $4\big|a^*_{2h}\big|$ since $r_{\widetilde{S}}(n)\le 1$ for all $n\in \mathbb{Z}$. Now, we remove all such pairs $s_1,s_2$ from $\widetilde{S}$ to form a new subset $S^*\subset \big[y/2,\left(1-\varepsilon/2\right)y\big)$. Clearly, from (\ref{eq-2-2}) we have
\begin{align}\label{eq-2-3}
\big|S^*\big|\ge \left(\frac{\sqrt{2}}{2}-\frac{\varepsilon}{2}\right)\sqrt{y}+O\big(y^{1/3}\big)-8a^*_{2h}\ge \left(\frac{\sqrt{2}}{2}-\varepsilon\right)\sqrt{y},
\end{align}
provided that $y$ is sufficiently large (in terms of $\varepsilon$ and $a^*_{2h}$). From now on, suppose that $y>x_h$ is a sufficiently large given number (in terms of $\varepsilon$ and $a^*_{2h}$),
then both of the following two equations
$$
a_1+a_2=s_1+s_2, \quad (a_1,a_2\in A_{2h}, ~s_1,s_2\in S^*)
$$
and 
$$
a_1+a_2=a_3+s_1, \quad (a_1,a_2,a_3\in A_{2h}, ~s_1\in S^*)
$$
have no solutions. Note further that we have
$$
s_1+s_2-s_3\ge y/2+y/2-\big(1-\varepsilon/2\big)y=\varepsilon y/2>a
$$
for any $s_1,s_2,s_3\in S^*$ and any $a\in A_{2h}$, which clearly means that
$$
s_1+s_2=s_3+a, \quad (a\in A_{2h}, ~s_1,s_2,s_3\in S^*)
$$
has no solutions. Now, let
$$
A_{2h+1}=A_{2h}\cup S^* \quad \text{and} \quad x_{h+1}=y.
$$
Then we have $r_{A_{2h+1}}(n)\le 1$ for all $n\in \mathbb{Z}$ and $x_{h+1}>x_h$ from the discussions above. Moreover, from (\ref{eq-2-3}) we have
$$
\frac{A_{2h+1}\big(-x_{h+1},x_{h+1}\big)}{\sqrt{x_{h+1}}}>\frac{\big|S^*\big|}{\sqrt{y}}\ge \frac{\sqrt{2}}{2}-\varepsilon.
$$
Thus, so far we have finished the constructions of $\{A_h\}$ and $\{x_h\}$.

Let
\begin{align*}
A=\bigcup_{h=1}^{\infty}A_h
\end{align*}
By the inductive hypotheses I and II, we have $r_{A_{h}}(n)=1$ for any $n\in \mathbb{Z}$. Hence, $A$ is a unique representation basis for the integers. By inductive hypothesis III, we have
$$
A(-x_h,x_h)\ge A_h(-x_h,x_h)\ge \left(\frac{\sqrt{2}}{2}-\varepsilon\right)\sqrt{x_h},
$$
which completes the proof of Theorem \ref{thm2}.
\end{proof}

The proof of Theorem \ref{thm3} is motivated by the idea of Erd\H os \cite{Stohr}, where he communicated to St\"ohr the argument for an old claim of Erd\H os and Tur\'an \cite{erdos-turan} on infinite Sidon sets.
 
\begin{proof}[Proof of Theorem \ref{thm3}]
Let $n$ be a sufficiently large integer. For any positive integer $\ell$, let $N_\ell$ (resp. $M_\ell$) be the number of elements of $A$ in the interval 
$$
\big((\ell-1)n,\ell n\big] \quad \left(\text{resp.} ~\big[-\ell n,(-\ell+1) n\big)\right).
$$
Suppose that $N_\ell\ge 2$ and $a,a'\in A\cap\big((\ell-1)n,\ell n\big]$ with $a<a'$, then
$$
0<a'-a<n.
$$
Since $A$ is a unique basis, the differences $a'-a$ are all different. Thus,
\begin{align*}
\sum_{\substack{1\le \ell\le n,~N_\ell\ge 2}}\binom{N_\ell}{2}< n,
\end{align*}
which means that 
\begin{align}\label{eq-3-1}
\sum_{\substack{1\le \ell\le n,~N_\ell\ge 2}}N_\ell^2\le 4\sum_{\substack{1\le \ell\le n,~N_\ell\ge 2}}\binom{N_\ell}{2}< 4n,
\end{align}
Moreover, we clearly have
\begin{align}\label{eq-3-2}
\sum_{\substack{1\le \ell\le n,~N_\ell=1}}N_\ell^2=\sum_{\substack{1\le \ell\le n,~N_\ell=1}}1\le n.
\end{align}
Hence, from (\ref{eq-3-1}) and (\ref{eq-3-2}) we get
\begin{align}\label{eq-3-3}
\sum_{\substack{1\le \ell\le n}}N_\ell^2<5n.
\end{align}
Similarly to the proof of (\ref{eq-3-3}), we have
\begin{align}\label{eq-3-4}
\sum_{\substack{1\le \ell\le n}}M_\ell^2<5n.
\end{align}
Now, we make a further observation which would be the main novelty of the proof. For  
$$
a\in A\cap\big((\ell-1)n,\ell n\big] \quad \text{and} \quad b\in A\cap\big[-\ell n,(-\ell+1) n\big),
$$
we have 
$$
-n<a+b<n.
$$
Since $A$ is a unique basis, the sums $a+b$ are all different. So we obtain
\begin{align}\label{eq-3-5}
N_\ell M_\ell< 2n.
\end{align}
Collecting together (\ref{eq-3-3}), (\ref{eq-3-4}) and (\ref{eq-3-5}), we get
\begin{align}\label{eq-3-6}
\sum_{\substack{1\le \ell\le n}}\left(N_\ell+ M_\ell\right)^2< 14n.
\end{align}

It is clear that
\begin{align}\label{eq-3-7}
\bigg(\sum_{1\le \ell\le n}\frac{N_\ell+M_\ell}{\ell^{1/2}}\bigg)^2\le \bigg(\sum_{1\le \ell\le n}\frac{1}{\ell}\bigg)\bigg(\sum_{1\le \ell\le n}(N_\ell+M_\ell)^2\bigg)
< (1+\log n) 14n
\end{align}
from the Cauchy--Schwarz inequality. By partial summations, we have
\begin{align}\label{eq-3-8}
\sum_{1\le \ell\le n}\frac{N_\ell+M_\ell}{\ell^{1/2}}&=\sum_{1\le \ell\le n}\Big(A\big(-\ell n,\ell n\big)-A\big(-(\ell-1) n,(\ell-1) n\big)\Big)\frac{1}{\ell^{1/2}}\nonumber\\
&=\sum_{1\le \ell\le n}A\big(-\ell n,\ell n\big)\left(\frac{1}{\ell^{1/2}}-\frac{1}{(\ell+1)^{1/2}}\right)+\frac{A\big(-n^2,n^2\big)}{(n+1)^{1/2}}\nonumber\\
&\ge \min_{1\le \ell\le n}\frac{A\big(-\ell n,\ell n\big)}{\sqrt{\ell n/\log (\ell n)}}\sum_{1\le \ell\le n}\sqrt{\frac{\ell n}{\log (\ell n)}}\left(\frac{1}{\ell^{1/2}}-\frac{1}{(\ell+1)^{1/2}}\right)\nonumber\\
&>\min_{1\le \ell\le n}\frac{A\big(-\ell n,\ell n\big)}{\sqrt{\ell n/\log (\ell n)}}\sqrt{\frac{n}{2\log n}}\sum_{1\le \ell\le n}\frac{1}{2(\ell+1)},
\end{align}
where we used the estimates
$$
\sqrt{\ell}\left(\frac{1}{\ell^{1/2}}-\frac{1}{(\ell+1)^{1/2}}\right)=\frac{\sqrt{\ell+1}-\sqrt{\ell}}{\sqrt{\ell+1}}=\frac{1}{\sqrt{\ell+1}(\sqrt{\ell+1}+\sqrt{\ell})}>\frac{1}{2(\ell+1)}
$$
in the last inequality. Inserting (\ref{eq-3-8}) into (\ref{eq-3-7}), we clearly have
\begin{align}\label{eq-3-9}
\min_{1\le \ell\le n}\frac{A\big(-\ell n,\ell n\big)}{\sqrt{\ell n/\log (\ell n)}}<\frac{4\sqrt{7}\sqrt{(1+\log n)\log n}}{\sum_{1\le \ell\le n}\frac{1}{\ell+1}}<4\sqrt{7}+\delta
\end{align}
for any $\delta>0$ providing that $n>n_\delta$, from which our theorem clearly follows.
\end{proof}

\section*{Acknowledgments}
We thank Doctors Shi--Qiang Chen and Yu--Chen Sun for their interests on this article.

The author is supported by National Natural Science Foundation of China  (Grant No. 12201544), Natural Science Foundation of Jiangsu Province, China (Grant No. BK20210784), China Postdoctoral Science Foundation (Grant No. 2022M710121).

\end{document}